\documentclass[12pt]{article}
\author{Will Johnson}
\title{Largeness and generalized t-henselianity}

\usepackage{amsmath, amssymb, amsthm}    	
\usepackage{fullpage} 	
\usepackage{amscd}
\usepackage{hyperref}
\usepackage[all]{xy}
\usepackage[T1]{fontenc}
\usepackage{lmodern}
\usepackage{centernot}
\usepackage[shortlabels]{enumitem}

\DeclareMathOperator*{\ind}{\raise0.2ex\hbox{\ooalign{\hidewidth$\vert$\hidewidth\cr\raise-0.9ex\hbox{$\smile$}}}}

\newcommand{\hens}{\mathrm{hens}}

\newcommand{\Gal}{\operatorname{Gal}}

\newcommand{\Frac}{\operatorname{Frac}}

\newcommand{\Spec}{\operatorname{Spec}}

\newcommand{\res}{\operatorname{res}}

\newtheorem{theorem}{Theorem}[section] 

\newtheorem{lemma}[theorem]{Lemma}

\newtheorem{corollary}[theorem]{Corollary}

\newtheorem{assumption}[theorem]{Assumption}

\newtheorem{question}[theorem]{Question}
\newtheorem{proposition}[theorem]{Proposition}
\newtheorem{proposition-eh}[theorem]{Proposition(?)}
\newtheorem*{theorem-star}{Theorem}
\newtheorem*{conjecture-star}{Conjecture}
\newtheorem*{lemma-star}{Lemma}

\theoremstyle{definition}
\newtheorem{fact}[theorem]{Fact}
\newtheorem{definition}[theorem]{Definition}
\newtheorem{example}[theorem]{Example}

\newtheorem{remark}[theorem]{Remark}

\theoremstyle{remark}
\newtheorem{claim}[theorem]{Claim}

\newtheorem*{acknowledgment}{Acknowledgments}

\newcommand{\Aa}{\mathbb{A}}
\newcommand{\Qq}{\mathbb{Q}}

\newcommand{\alg}{\mathrm{alg}}

\newcommand{\Ee}{\mathcal{E}}
\newcommand{\Rr}{\mathbb{R}}

\newcommand{\Nn}{\mathbb{N}}
\newcommand{\Cc}{\mathbb{C}}

\newcommand{\Ff}{\mathbb{F}}

\newcommand{\Ll}{\mathcal{L}}

\newcommand{\mm}{\mathfrak{m}}

\newcommand{\ba}{{\bar{a}}}
\newcommand{\bb}{{\bar{b}}}

\newenvironment{claimproof}[1][\proofname]
               {
                 \proof[#1]
                 
               }
               {
                 \endproof
               }

\let\phi\varphi

\begin{document}

\maketitle

\begin{abstract}
  Let $K$ be a countable field.  Then $K$ is large in the sense of Pop~\cite{Pop-little}
  if and only if it admits a field topology which is ``generalized
  t-henselian'' (gt-henselian) in the sense of Dittmann, Walsberg, and
  Ye~\cite{field-top-2}, meaning that the implicit function theorem holds for
  polynomials.  Moreover, the \'etale open topology of \cite{firstpaper} can be
  characterized in terms of the gt-henselian topologies on $K$: a
  subset $U \subseteq K^n$ is open in the \'etale open topology if and
  only if it is open with respect to every gt-henselian topology on
  $K$.
\end{abstract}

\section{Introduction}
A field $K$ is \emph{large} \cite{Pop-little} if it satisfies the
following two equivalent conditions:
\begin{enumerate}
\item If $C$ is a smooth algebraic curve over $K$, and $C(K) \ne
  \varnothing$, then $C(K)$ is infinite.  Here, $C(K)$ denotes the set
  of $K$-rational points on $C$.  
\item Let $f(X,Y) \in K[X,Y]$ be a polynomial such that $f(0,0) = 0
  \ne \frac{\partial f}{\partial Y}(0,0)$, as in the implicit function
  theorem.  Then $\{(a,b) \in K^2 : f(a,b) = 0\}$ is infinite.
\end{enumerate}
There are many examples of large fields, including separably closed
fields, real closed fields, $p$-adically closed fields, fields
admitting non-trivial henselian valuations, pseudo finite fields, and
pseudo algebraically closed (PAC) fields.

Pop~\cite{Pop-henselian} showed that if $R$ is a henselian local
domain which is not a field, then $\Frac(R)$ is large.  In
\cite{large-hens}, we showed that all large fields arise this way, up
to elementary equivalence.  The present note is an addendum to
\cite{large-hens}, showing another connection between largeness and
``henselianity'' construed broadly.

In this paper, field topologies are assumed to be Hausdorff and non-discrete.  A field topology $\tau$ on $K$ is
\emph{generalized t-henselian} or \emph{gt-henselian}
\cite[Definition~8.1]{field-top-2} if it satisfies the following
equivalent conditions:
\begin{enumerate}
\item The inverse function theorem holds for polynomials: if $\bar{P} : K^n
  \to K^n$ is a polynomial map whose Jacobian matrix $\{\partial P_i /
  \partial X_j\}_{i,j}$ is invertible at $\bar{0} \in K^n$, then $P$
  is a local homeomorphism at $\bar{0}$, with respect to $\tau$.
\item The implicit function theorem holds for polynomials: if
  $\bar{P}(\bar{X},\bar{Y})$ is a polynomial map $K^n \times K^m \to
  K^m$ such that $\bar{P}(\bar{0},\bar{0}) = \bar{0}$ and the matrix $\{\partial P_i/\partial Y_j\}_{i,j}$ is invertible at $\bar{0}$, then there are $\tau$-neighborhoods $U \ni \bar{0}$ and $V \ni \bar{0}$ such that the set
  \begin{equation*}
    \{(\bar{a},\bar{b}) \in U \times V : \bar{P}(\bar{a},\bar{b}) = \bar{0}\}
  \end{equation*}
  is the graph of a $\tau$-continuous function $f : U \to V$.
\item If $f : V \to W$ is an \'etale morphism of $K$-varieties, then
  $V(K) \to W(K)$ is a local homeomorphism with respect to $\tau$.
\item For any $n \ge 2$ and neighborhood $U \ni -1$, there is a
  neighborhood $V \ni 0$ such that if $c_0,\ldots,c_{n-2} \in V$ then
  the polynomial $X^n + X^{n-1} + c_{n-2}X^{n-2} + \cdots + c_1X +
  c_0$ has a root in $U$.
\end{enumerate}
See \cite[Proposition~6.2]{tops-rings} for a proof of the equivalence,
which mostly comes from \cite[\S 8]{field-top-2}.  Generalized t-henselianity
first appeared in work of Pop under a different name.\footnote{If
$\tau$ is a field topology on $K$, then $\tau$ is gt-henselian
if and only if $K$ is $\tau$-henselian in the sense of Pop (see
discussion after \cite[Theorem~1.6]{Pop-little}).}
\begin{fact} \phantomsection \label{vest}
  \begin{enumerate}
  \item \label{v1} If $\tau$ is a gt-henselian topology on $K$, then $K$ is large
    \cite[Corollary~8.15]{field-top-2}, essentially by the implicit
    function theorem.
  \item \label{v2} If $R$ is a henselian local domain and $K = \Frac(R) \ne R$,
    then the family of sets $\{aR + b : a,b \in K, ~ a \ne 0\}$ is a
    basis for a gt-henselian topology on $R$
    \cite[Proposition~8.5]{field-top-2}.
  \item A field topology $\tau$ is \emph{topologically henselian} (\emph{t-henselian}) in the sense
  of Prestel and Ziegler \cite[\S 7]{PZ} if and only if $\tau$ is a V-topology (meaning that $\tau$ is
    induced by an absolute value or valuation~\cite[Appendix~B]{PE}) and $\tau$ is gt-henselian
    \cite[Proposition~8.3]{field-top-2}.
  \item The following topological fields are
    t-henselian, hence gt-henselian:
    \begin{itemize}
    \item The valuation topology on a henselian valued field
      \cite[Definition~7.1]{PZ}.
    \item The standard ``analytic'' topologies on $\Rr, \Cc, \Qq_p$
      \cite[Corollary~7.3]{PZ}.
    \item The order topology on real closed fields \cite[Theorem~6.1
      and Corollary~7.3]{PZ}.
    \end{itemize}
  \end{enumerate}
\end{fact}
\noindent Our main theorem is a partial converse to Fact~\ref{vest}(\ref{v1}):
\begin{theorem}[{= Theorem~\ref{exists}$+$Fact~\ref{vest}(\ref{v1})}] \label{main-thm}
  Let $K$ be a countable field.  Then $K$ is large if and only if $K$
  admits a gt-henselian field topology.
\end{theorem}
Although many examples of large fields carry obvious gt-henselian
topologies, many others do not, such as the following:
\begin{itemize}
\item Infinite algebraic extensions of $\Ff_p$, such as $\Ff_p^{\alg}$.  Such extensions are PAC.
\item Pseudo finite fields, such as an ultraproduct of finite fields,
  or the fixed field of a random automorphism $\sigma \in
  \Gal(\Qq^{\alg}/\Qq)$ (see \cite[Theorem~20.5.1]{field-arithmetic}).
\end{itemize}
In fact, these fields don't admit t-henselian topologies,\footnote{An
algebraic extension $K/\Ff_p$ has no non-trivial valuations or
absolute values, so it admits no V-topologies.  If a pseudo finite
field admitted a t-henselian topology, the proof of
\cite[Remark~7.11]{PZ} would show that some finite field admits a
t-henselian topology.  But finite fields do not admit field
topologies} which makes it surprising that they admit
\emph{gt}-henselian topologies.

Theorem~\ref{main-thm} raises the following natural question, which I
was unable to answer:
\begin{question}
  Does Theorem~\ref{main-thm} generalize to uncountable fields?  In
  other words, does every large field admit a gt-henselian field
  topology?
\end{question}
\noindent Model-theoretically, Theorem~\ref{main-thm} has the following
consequence:
\begin{corollary} \label{down}
  A field $(K,+,\cdot)$ is large if and only there is an elementary
  substructure $(K_0,+,\cdot) \preceq (K,+,\cdot)$ admitting a
  gt-henselian topology.
\end{corollary}
\begin{proof}
  If $K$ is large, the downward L\"owenheim-Skolem theorem gives a
  countable elementary substructure $K_0 \preceq K$.  Largeness is
  preserved under elementary equivalence, so $K_0$ is large and admits
  a gt-henselian topology by Theorem~\ref{main-thm}.  Conversely,
  suppose $K_0 \preceq K$ and $K_0$ admits a gt-henselian topology
  $\tau$.  Then $K_0$ is large by Fact~\ref{vest}(\ref{v1}) and so $K$
  is large.
\end{proof}
\noindent Corollary~\ref{down} is dual to the following consequence of
\cite{large-hens}:
\begin{proposition} \label{up}
  A field $(K,+,\cdot)$ is large if and only if there is an elementary
  extension $(K_1,+,\cdot) \succeq (K,+,\cdot)$ admitting a locally
  bounded gt-henselian topology.
\end{proposition}
\noindent See \cite[\S 2]{PZ} for background on locally bounded field
topologies.
\begin{proof}
  If $K$ is large, then \cite[Remark~8.9]{large-hens} shows that there
  is an elementary extension $K_1$ such that $K_1 = \Frac(R)$ for some
  proper subring $R \subsetneq K_1$ such that $R$ is a henselian local
  domain.  By Fact~\ref{vest}(\ref{v2}), $R$ induces a gt-henselian
  topology on $K_1$, which is locally bounded
  \cite[Theorem~2.2(a)]{PZ}.  The rest of the proof proceeds like
  Corollary~\ref{down}.
\end{proof}
\begin{remark} \label{no-bound}
  In Theorem~\ref{main-thm} and Corollary~\ref{down}, we cannot ask
  for the gt-henselian topology to be locally bounded, because
  $\Ff_p^{\alg}$ and its subfields admit no locally bounded field
  topologies.  This fact is presumably well-known, and easy to prove
  as follows.  Suppose $K/\Ff_p$ is algebraic and $\tau$ is a locally
  bounded field topology on $K$.  By \cite[Lemma~2.1(e)]{PZ}, there is
  an open neighborhood $U \ni 0$ such that $\{cU : c \in K^\times\}$
  is a neighborhood basis of 0.  Take some $c \in K^\times$ such that
  $cU \subsetneq U$.  Then $c^n U \subsetneq U$ for every $n \ge 1$.
  But $c^{p^k} = 1$ for $k = [\Ff_p(c) : \Ff_p]$, giving the absurd
  conclusion $U = c^{p^k} U \subsetneq U$.
\end{remark}
By being more careful with the proof of Theorem~\ref{main-thm}, we can
also say something about the \emph{\'etale-open topology} $\Ee$ over
$K$, a key technical tool used in \cite{large-hens} and an earlier
version of the present note.  We first recall the definition of the
\'etale-open topology:
\begin{definition}
  If $V$ is a $K$-variety, an \emph{E-set} on $V$ is a subset of
  $V(K)$ of the form $f(W(K))$ for some \'etale morphism $f : W \to V$
  of $K$-varieties.  The \emph{\'etale-open topology} or
  \emph{$\Ee$-topology} on $V(K)$ is the topology with E-sets as basic open sets.
\end{definition}
We prove the following characterization of the $\Ee$-topology in terms
of gt-henselian topologies:
\begin{theorem}[{= Theorem~\ref{2m}}] \label{second-main}
  Let $K$ be countable.  If $V$ is a $K$-variety, a subset $U
  \subseteq V(K)$ is $\Ee$-open if and only if $U$ is $\tau$-open for
  every gt-henselian field topology $\tau$ on $K$.
\end{theorem}
\noindent Again, it would be nice to know whether Theorem~\ref{second-main}
holds for uncountable fields.

\subsection{A note on the proofs}
Kiltinen \cite{kiltinen} showed that any countable field $K$ admits a
field topology, and Podewksi \cite{podewski} showed that $K$ in fact
admits $2^{2^{\aleph_0}}$-many distinct field topologies.  If I
understand correctly, the strategy (described in \cite[\S
  1]{kiltinen}) for building topologies is to inductively build an array of finite subsets of
$K$
\begin{gather*}
  \begin{pmatrix}
    V^0_0 & V^0_1 & V^0_2 & \cdots \\
    V^1_0 & V^1_1 & V^1_2 & \cdots \\
    V^2_0 & V^2_1 & V^2_2 & \cdots \\
    \vdots & \vdots & \vdots & \ddots
  \end{pmatrix}
\end{gather*}
such that
\begin{itemize}
\item The rows are decreasing ($V^j_i \supseteq V^j_{i+1}$), and each
  row ends with finitely many copies of $\{0\}$.
\item The columns are increasing ($V^j_i \subseteq V^{j+1}_i$).
\item The $j$th row satisfies the Conditions that would make
  $\{V^j_i : i \in \Nn\}$ be a neighborhood basis for a field
  topology, except for the non-discreteness condition that $V^j_i \ne
  \{0\}$.
\item Row $j+1$ is obtained from row $j$ by adding one new element in
  column $j$, and closing under the Conditions.
\end{itemize}
Then one takes $V^\omega_i = \bigcup_{j=0}^\infty V^j_i$.  The
Conditions on $V^j_\bullet$  transfer to $V^\omega_\bullet$ in
the limit, and the fact that we added one new element in each row
ensures that each set $V^\omega_\bullet$ is strictly bigger than
$\{0\}$.  Then $\{V^\omega_i : i \in \Nn\}$ is a neighborhood basis
for a field topology on $K$.

Our construction works in a similar way, but adding some additional
Conditions for gt-henselianity.  These new Conditions complicate the
construction of row $j+1$ from row $j$.  For example, one of the
conditions might say something like ``if $x \in V_{23}$, then
$\sqrt{1+x}-1 \in V_{22}$'', but we don't have a unique choice of a
square root.  The original plan was to use the machinery of \emph{Nash
functions} from \cite[Section~7]{large-hens}, building off the
fact that \'etale morphisms usually induce local homeomorphisms in the
$\Ee$-topology \cite[Theorem~B]{large-hens}.  However, this
machinery fails for $\Ff_p^\alg$, and so we use a slightly different
proof based on the following theorem of F.-V. Kuhlmann \cite[Theorem~15]{kuhlmann}
\begin{fact} \label{Kuhlmann}
  A field $K$ is large if and only if $K$ is existentially closed in
  the field of Laurent series $K((t))$.
\end{fact}

\subsection{Conventions}

As noted in the introduction, field topologies are assumed to be
Hausdorff and non-discrete.  We let $\Ll_{rings}$ denote the language of rings $\{+,-,\cdot,0,1\}$.  If $\Ll$ is a language and $A$ is a subset of an $\Ll$-structure, then $\Ll(A)$ denotes the language obtained by adding the elements of $A$ as new constant symbols.  If $X,Y$ are subsets of a ring, then
\begin{gather*}
	X-Y = \{x - y : x \in X, ~ y \in Y\} \\
	X \cdot Y = \{xy : x \in X, ~ y \in Y\} \\
	X + Y = \{x+y : x \in X, ~ y \in Y\}.
\end{gather*}
A \emph{variety} over $K$ means a
reduced separated finite-type scheme over $K$.  If $V$ is a variety
over $K$, then $V(K)$ denotes the set of $K$-rational points on $V$,
i.e., morphisms $\Spec K \to V$.  If $V$ is an affine variety over
$K$, then $K[V]$ denotes the coordinate ring of $V$, i.e., the
$K$-algebra whose spectrum is $V$.

We will need the following variant of the definition of
gt-henselianity:
\begin{lemma}\label{upgrade}
  Let $(K,\tau)$ be a gt-henselian field topology.  For any $n \ge 2$
  and neighborhood $U$ of $-1$, there is a neighborhood $V$ of $0$
  such that if $c_0,\ldots,c_{n-2} \in V$, then the polynomial $X^n +
  X^{n-1} + c_{n-2}X^{n-2} + \cdots + c_0$ has a \underline{simple}
  root in $U$.
\end{lemma}
This is the same as the definition of gt-henselianity in \cite[Definition~8.1]{field-top-2}, except that the
word ``simple'' is added.
\begin{proof}
  Let $f$ and $g$ be the following polynomials:
  \begin{align*}
    f(X,Z_0,\ldots,Z_{n-2}) &= X^n + X^{n-1} + Z_{n-2}X^{n-2} + \cdots + Z_1X + Z_0 \\
    g(X,Z_0,\ldots,Z_{n-2}) &= \frac{\partial}{\partial X}f \\ &= nX^{n-1} + (n-1)X^{n-2} + (n-2)Z_{n-2}X^{n-3} + \cdots + 2Z_2 X + Z_1.
  \end{align*}
  Note that $g(-1,0,0,\ldots,0) = n(-1)^{n-1} + (n-1)(-1)^{n-2} = \pm
  1$.  By continuity of the ring operations, there are neighborhoods
  $U_0$ of $-1$ and $V_0$ of $0$ such that if $a \in U_0$ and
  $c_0,\ldots,c_{n-2} \in V_0$, then
  \begin{equation*}
    g(a,c_0,\ldots,c_{n-2}) \ne 0.
  \end{equation*}
  Given a neighborhood $U$ of $-1$, gt-henselianity gives a
  neighborhood $V_1$ of 0 such that if $c_0,\ldots,c_{n-2} \in V_1$,
  then the polynomial $f(X,c_0,\ldots,c_{n-2}) = X^n + X^{n-1} +
  c_{n-2}X^{n-2} + \cdots + c_0$ has a root $a$ in $U \cap U_0$.  Take
  $V = V_1 \cap V_0$.  Suppose $c_0,\ldots,c_{n-2} \in V$.  Then
  $f(X,c_0,\ldots,c_{n-2})$ has a root $a \in U \cap U_0$.  Since $a
  \in U_0$ and $c_0,\ldots,c_{n-2} \in V_0$, it follows that
  \begin{equation*}
    f(a,c_0,\ldots,c_{n-2}) = 0 \ne g(a,c_0,\ldots,c_{n-2}).
  \end{equation*}
  Since $g(X,c_0,\ldots,c_{n-2})$ is the derivative of
  $f(X,c_0,\ldots,c_{n-2})$, it follows that $a$ is a simple root of
  $f(X,c_0,\ldots,c_{n-2})$.
\end{proof}

\section{Suitable sequences}
For the remainder of the paper, fix a countable large field $K$.
Recall that a field topology on $K$ is determined by specifying a
neighborhood basis of 0.
\begin{fact} \label{ofcourse}
  Let $\tau$ be a non-empty family of subsets of $K$.  Then $\tau$ is
  a neighborhood basis for a gt-henselian topology on $K$ if and only
  if the following conditions hold:
  \begin{enumerate}
  \item \label{pz1} For any $U, V \in \tau$, there is $W \in \tau$ with $W
    \subseteq U \cap V$.
  \item For any $U \in \tau$, we have $U \supsetneq \{0\}$.
  \item For any $U \in \tau$, there is $V \in \tau$ with $V - V
    \subseteq U$.
  \item For any $U \in \tau$, there is $V \in \tau$ with $V \cdot V
    \subseteq U$.
  \item There is $U \in \tau$ with $-1 \notin U$.
  \item For any $a \in K^\times$ and $U \in \tau$, there is $V \in
    \tau$ with $a \cdot V \subseteq U$.
  \item \label{pzom} For any $U \in \tau$, there is $V \in \tau$ with $(1 + V)^{-1}
    \subseteq 1 + U$.
  \item \label{gt} For any $n \ge 2$ and any $U \in \tau$, there is $V \in \tau$
    such that if $c_0, \ldots, c_{n-2} \in V$ then the polynomial $X^n
    + X^{n-1} + c_{n-2}X^{n-2} + \cdots + c_1X + c_0$ has a simple
    root in $-1 + U$.
  \end{enumerate}
\end{fact}
\begin{proof}
  Conditions~(\ref{pz1})--(\ref{pzom}) express that $\tau$ is a
  neighborhood basis for a field topology; see
  \cite[Introduction]{PZ}.  The last condition is just
  Lemma~\ref{upgrade}.
\end{proof}
For the remainder of the paper,
fix an enumeration $\{a_i\}_{i \in \Nn}$ of the elements of $K$ such
that each element appears infinitely many times.  Also fix a sequence
$\{n_i\}_{i \in \Nn}$ which lists each element of $\{2,3,4,5,\ldots\}$
infinitely many times.
\begin{definition}
  Let $R$ be an integral $K$-algebra.  Let $A, B$ be subsets of $R$.
  Then $C_i(A,B)$ is the conjunction of the following statements:
  \begin{enumerate}
  \item $0 \in A$.
  \item $B \subseteq A$.
  \item $B - B \subseteq A$.
  \item $B \cdot B \subseteq A$.
  \item $a_i \cdot B \subseteq A$.
  \item $-1 \notin B$, and $(1+B)^{-1} \subseteq 1 + A$.
  \item If $c_0,\ldots,c_{n_i - 2} \in B$, then the polynomial
    \begin{equation*}
      X^{n_i} + X^{n_i - 1} + c_{n_i - 2}X^{n_i - 2} + \cdots + c_1 X^1 + c_0
    \end{equation*}
    has at least one simple root in the set $-1 + A$.
  \end{enumerate}
\end{definition}
\begin{remark} \label{easy-props}
  The following properties of $C_i$ can be checked by inspection:
  \begin{enumerate}
  \item \label{zero} $C_i(\{0\},\{0\})$ is true.
  \item $C_i(A,B)$ depends positively on $A$ and negatively on $B$.
    In other words,
    \begin{equation*}
      (C_i(A,B) \text{ and } A' \supseteq A \text{ and } B' \subseteq
      B) \implies C_i(A',B').
    \end{equation*}
\item \label{uni} If $A^0 \subseteq A^1 \subseteq \cdots$ is an increasing
  sequence and $B^0 \subseteq B^1 \subseteq \cdots$ is an increasing
  sequence and $C_i(A^j,B^j)$ holds for each $j$, then
  $C_i\left(\bigcup_{j=0}^\infty A^j, \bigcup_{j=0}^\infty B^j\right)$
  holds.
  \end{enumerate}
\end{remark}
\begin{definition}
  Suppose $R$ is an integral domain extending $K$ and $\{A_i\}_{i \in
    \Nn}$ is a sequence of subsets of $R$.  Then $\{A_i\}_{i \in \Nn}$
  is \emph{suitable} if $C_i(A_i,A_{i+1})$ holds for each $i$.
\end{definition}
For example,
\begin{itemize}
\item $A_i \supseteq A_{i+1}$ holds for each $i$, so the sequence is
  decreasing.
\item $A_i \supseteq A_{i+1} - A_{i+1}$ for each $i$.
\item $-1 \notin A_{i+1}$ for $i \ge 0$.
\item $0 \in A_i$ for $i \ge 0$.
\item $A_i \supseteq a_i A_{i+1}$.
\end{itemize}
and so on.
\begin{remark} \label{zero2}
  The sequence $(\{0\},\{0\},\ldots)$ is suitable, by
  Remark~\ref{easy-props}(\ref{zero}).
\end{remark}
The significance of suitable sequences is that they give
gt-henselian topologies, when the sets are sufficiently big:
\begin{lemma} \label{obvious}
  Let $\{A_i\}_{i \in \Nn}$ be a suitable sequence of subsets of
  $K$.  Then one of the following holds:
  \begin{itemize}
  \item $A_i = \{0\}$ for all sufficiently large $i$.
  \item The family $\{A_i : i \in \Nn\}$ is a neighborhood basis for a
    gt-henselian topology on $K$.
  \end{itemize}
\end{lemma}
The proof is probably obvious, but we give the details for
completeness.
\begin{proof}
  Each $A_i \supseteq \{0\}$.  If some $A_i = \{0\}$, then $A_i =
  A_{i+1} = \cdots = \{0\}$, since the sequence is decreasing, and
  then the first point holds.  Therefore, we may assume no $A_i$
  equals $\{0\}$.

  We show that the conditions of Fact~\ref{ofcourse} hold.
  \begin{enumerate}
  \item The sequence $A_0, A_1, A_2, \ldots$ is decreasing, so the set
    $\{A_i : i \in \Nn\}$ is downward directed.
  \item By assumption, each $A_i \supsetneq \{0\}$.  Otherwise, we are
    in the first case.
  \item For each $i$, $A_i \supseteq A_{i+1} - A_{i+1}$ by definition
    of $C_i$.
  \item Similarly, $A_i \supseteq A_{i+1} \cdot A_{i+1}$.
  \item By definition of $C_0$, the condition $C_0(A_0,A_1)$ implies
    $-1 \notin A_1$.
  \item Fix some $a \in K^\times$ and some $i$.  We must find some $j$
    such that $a \cdot A_j \subseteq A_i$.  Increasing $i$, we may
    suppose $a_i = a$.  Then $C_i(A_i,A_{i+1})$ implies $a_i \cdot
    A_{i+1} \subseteq A_i$, so we may take $j = i+1$.
  \item $C_i(A_i,A_{i+1})$ implies $(1 + A_{i+1})^{-1} \subseteq 1 +
    A_i$.
  \item Fix some $n \ge 2$ and $i$.  We must find some $j$ such that
    if $c_0,\ldots,c_{n-2} \in A_j$, then the polynomial $X^n +
    X^{n-1} + c_{n-2}X^{n-2} + \cdots + c_0$ has a simple root in $-1
    + A_i$.  Increasing $i$, we may suppose $n_i = n$.  Then
    $C_i(A_i,A_{i+1})$ shows that we can take $j = i+1$.  \qedhere
  \end{enumerate}
\end{proof}
\begin{lemma} \label{union-duh}
  Let $A^j_i$ be a family of subsets of $K$ such that\ldots
  \begin{itemize}
  \item For fixed $j$, the sequence $\{A^j_i\}_{i \in \Nn}$ is suitable.
  \item For fixed $i$, the sequence $\{A^j_i\}_{j \in \Nn}$ is
    increasing: $A_i^0 \subseteq A_i^1 \subseteq A_i^2 \subseteq
    \cdots$.
  \end{itemize}
  Let $A_i^\omega = \bigcup_{j=0}^\infty A_i^j$.  Then the sequence
  $\{A^\omega_i\}_{i \in \Nn}$ is suitable.
\end{lemma}
\begin{proof}
  Immediate from Remark~\ref{easy-props}(\ref{uni}).
\end{proof}
As in \cite{kiltinen,podewski}, our strategy will be to build an
increasing sequence of suitable sequences $A_\bullet^0 \subseteq
A_\bullet^1 \subseteq A_\bullet^2 \subseteq \cdots$, such that each
$A_\bullet^j$ falls into the first case of Lemma~\ref{obvious}, but
the union falls into the second case, giving a gt-henselian topology
on $K$.

\section{Adding finitely many new elements}
\begin{remark} \label{forge}
  There is a local $K$-algebra $(R,\mm)$ with the following properties:
\begin{itemize}
\item $R$ is an integral domain, but not a field, so that $\mm
  \supsetneq \{0\}$.
\item $K$ is the residue field of $R$, so that $R = K \oplus \mm$.
\item $K$ is existentially closed in $R$.
\item $R$ is henselian.
\end{itemize}
For example, we could take $R = K[[t]]$ or the henselization $K[t]_{(t)}^{\hens}$ of the local ring $K[t]_{(t)}$.
The field $K$ is existentially closed in both these rings because it's
existentially closed in the larger ring $K((t)) \supseteq K[[t]] \supseteq K[t]_{(t)}^\hens \supseteq K$ by
Fact~\ref{Kuhlmann}.
\end{remark}
Fix an $(R,\mm)$ as in Remark~\ref{forge}.
If $A \subseteq K$, note that $A + \mm$ is the set of elements of $R$
with residue in $A$.  Let $\res : R \to K$ denote the residue map.

The next lemma is the core of the proof of the main theorem:
\begin{lemma} \label{core}
  Suppose $A, B$ are finite subsets of $K$ and $C_i(A,B)$ holds.
  Suppose $B'$ is a finite subset of $R$ with $B \subseteq B'
  \subseteq B + \mm$.  Then there is a finite subset $A' \subseteq R$
  with $A \subseteq A' \subseteq A + \mm$ and $C_i(A',B')$.
\end{lemma}
\begin{proof}
  \begin{claim} \label{flip}
    If $x \in B'$, then $1+x \in R^\times$, and $\res(1/(1+x)-1) \in
    A$.
  \end{claim}
  \begin{claimproof}
    By condition $C_i(A,B)$, we have $-1 \notin B$.  Because $B'
    \subseteq B + \mm$, we have $\res(x) \in B$ and so $\res(x) \ne
    -1$.  Then $\res(1 + x) = 1 + \res(x) \ne 0$, and so $1 + x \in
    R^\times$ and $1/(1+x)$ exists.  Also,
    \begin{equation*}
      \res(1/(1 + x) - 1) = \res(1/(1+x)) - 1 = \res(1+x)^{-1} - 1 =
      (1 + \res(x))^{-1} - 1.
    \end{equation*}
    Since $\res(x) \in B$ and $(1+B)^{-1} - 1 \subseteq A$, we have
    $\res(1/(1+x) - 1) \in A$.
  \end{claimproof}
  Define the following sets:
  \begin{enumerate}
  \item $Z_1 = \{0\}$.
  \item $Z_2 = B'$.
  \item $Z_3 = B' - B'$.
  \item $Z_4 = B' \cdot B'$.
  \item $Z_5 = a_i \cdot B'$.
  \item $Z_6 = (1 + B')^{-1} - 1$.  This makes sense by
    Claim~\ref{flip}.
  \item $Z_7$ is the set of $r \in A + \mm$ such that $-1 + r$ is a
    simple root of some polynomial of the form $X^{n_i} + X^{n_i - 1}
    + c_{n_i - 2}X^{n_i - 2} + \cdots + c_1X + c_0$ with
    $c_0,\ldots,c_{n_i - 2} \in B'$.
  \end{enumerate}
  \begin{claim}
    Each $Z_j$ is finite.
  \end{claim}
  \begin{claimproof}
    Clear, since $B'$ is finite.
  \end{claimproof}
  \begin{claim}
    Each $Z_j$ is a subset of $A + \mm$.
  \end{claim}
  \begin{claimproof}
    In other words, we must show that if $x \in Z_j$ then $\res(x) \in
    A$.
    \begin{enumerate}
    \item $Z_1$: this holds because $\res(0) = 0 \in A$ by $C_i(A,B)$.
    \item $Z_2$: if $x \in B'$, then $\res(x) \in B \subseteq A$ by
      $C_i(A,B)$.
    \item $Z_3$: if $x,y \in B'$, then $\res(x-y) = \res(x) - \res(y)
      \in B - B \subseteq A$.
    \item $Z_4$: if $x,y \in B'$, then $\res(xy) = \res(x)\res(y) \in
      B \cdot B \subseteq A$.
    \item $Z_5$: if $x \in B'$, then $\res(a_i x) = a_i \res(x) \in
      a_i B \subseteq A$.
    \item $Z_6$: if $x \in B'$, then $\res(1/(1+x)-1) \in A$ by
      Claim~\ref{flip}.
    \item $Z_7$: true by definition of $Z_7$.  \qedhere
    \end{enumerate}
  \end{claimproof}
  Take $A' = A \cup \bigcup_{j=1}^7 Z_j$.  Then $A'$ is a finite
  subset of $A + \mm$.  It remains to show that $C_i(A',B')$ holds.
  \begin{enumerate}
  \item $0 \in A'$ because $Z_1 = \{0\} \subseteq A'$.  Alternatively,
    $0 \in A'$ because $0 \in A$ by $C_i(A,B)$.
  \item $B' = Z_2 \subseteq A'$.
  \item $B' - B' = Z_3 \subseteq A'$.
  \item $B' \cdot B' = Z_4 \subseteq A'$.
  \item $a_i \cdot B' = Z_5 \subseteq A'$.
  \item $(1 + B')^{-1} - 1 = Z_6 \subseteq A'$.
  \item Finally, suppose $c_0, \ldots,c_{n_i - 2}$ are elements of
    $B'$.  Then $\res(c_0),\ldots,\res(c_{n_i - 2})$ are elements of
    $B$.  By the assumption $C_i(A,B)$, the polynomial
    \begin{equation*}
      X^{n_i} + X^{n_i - 1} + \res(c_{n_i- 2})X^{n_i - 2} + \cdots + \res(c_1)X^1 + \res(c_0)
    \end{equation*}
    has a simple root $\rho \in -1 + A$.  Because $R$ is henselian, we
    can lift $\rho$ to a simple root $r \in \rho + \mm$ of the polynomial
    \begin{equation*}
      X^{n_i} + X^{n_i - 1} + c_{n_i - 2}X^{n_i - 2} + \cdots + c_1X +
      c_0. \tag{$\ast$}
    \end{equation*}
    Then $\res(r+1) = \rho + 1 \in A$, and $-1 + (r+1)$ is a
    simple root of the polynomial ($\ast$), so $r+1 \in Z_7 \subseteq
    A'$.  Then $r \in -1 + A'$, and the polynomial ($\ast$) has at
    least one simple root in $-1 + A'$, as desired.  \qedhere
  \end{enumerate}
\end{proof}
\begin{definition}
  A suitable sequence $A_0, A_1, A_2, \ldots$ is \emph{essentially
  finite} if each $A_i$ is finite, and $A_i = \{0\}$ for $i \gg 0$.
\end{definition}
\begin{lemma} \label{iterate}
  Let $A_0, A_1, A_2, \ldots$ be an essentially finite suitable
  sequence in $K$.  Given any $n$, there is an essentially finite
  suitable sequence $A'_0, A'_1, \ldots$ in $R$ such that $A'_i
  \supseteq A_i$ for each $i$, and $A'_n \ne \{0\}$.
\end{lemma}
\begin{proof}
  Choose some non-zero ``seed'' $t \in \mm$.  Let $A'_i = A_i$ for $i
  > n$.  Let $A'_n = A_n \cup \{t\}$.  Then $C_i(A'_i,A'_{i+1})$ holds
  for all $i \ge n$.  Moreover, $A'_n$ is finite and $A_n \subseteq
  A'_n \subseteq A_n + \mm$ (because $\res(t) = 0 \in A_n$).

  Build $A'_{n-1}, A'_{n-2}, \ldots, A'_0$ recursively using
  Lemma~\ref{core}, ensuring that
  \begin{itemize}
  \item Each $A'_i$ is finite
  \item $A_i \subseteq A'_i \subseteq A_i + \mm$.
  \item $C_i(A_i,A_{i+1})$.
  \end{itemize}
  Then we are done.
\end{proof}
\begin{lemma} \label{ec-trick}
  Let $A_0, A_1, A_2, \ldots$ be an essentially finite suitable
  sequence in $K$.  Given any $n$, there is an essentially finite
  suitable sequence $A'_0, A'_1, \ldots$ in $K$ such that $A'_i
  \supseteq A_i$ for each $i$, and $A'_n \supsetneq A_n$.  
\end{lemma}
This is identical to Lemma~\ref{iterate} but with $R$ replaced by $K$.
\begin{proof}
  This follows directly from Lemma~\ref{iterate} because $K$ is
  existentially closed in $R$, and the sets involved are finite.

  In more detail, first let $\{A''_i\}_{i \in \Nn}$ be a suitable
  sequence in $R$ produced by Lemma~\ref{iterate}.  Let $m$ be so large that
  \begin{itemize}
  \item $|A''_i| \le m$ for every $i$.
  \item $A''_i = \{0\}$ for $i > m$.
  \end{itemize}
  Note the following:
  \begin{claim}
    For any finite $S \subseteq K$, there are quantifier-free
    $\mathcal{L}_{rings}(K)$-formulas $\theta_S(x_1,\ldots,x_m)$ and
    $\theta'_S(x_1,\ldots,x_m)$ expressing the conditions
    \begin{gather*}
      S \subseteq \{x_1,\ldots,x_m\} \\
      S \subsetneq \{x_1,\ldots,x_m\}.
    \end{gather*}
    respectively.
  \end{claim}
  \begin{claimproof}
    For example, we can express the second condition as
    \begin{gather*}
      \bigwedge_{s \in S} \bigvee_{i = 1}^m (s = x_i) \\
      \wedge \neg \bigwedge_{i=1}^m \bigvee_{s \in S} (x_i = s).
    \end{gather*}
    (The first line says $S \subseteq \{x_1,\ldots,x_m\}$ and the
    second says $\{x_1,\ldots,x_m\} \not \subseteq S$.)
  \end{claimproof}
  \begin{claim}
    For any $i$, there is a quantifier-free
    $\mathcal{L}_{rings}(K)$-formula
    $\phi_i(x_1,\ldots,x_m;y_1,\ldots,y_m)$ which expresses the
    condition $C_i(\{x_1,\ldots,x_m\},\{y_1,\ldots,y_m\})$ is any
    integral $K$-algebra.
  \end{claim}
  \begin{claimproof}
    Straightforward.  For example, to express that
    \begin{equation*}
      \{y_1,\ldots,y_m\} - \{y_1,\ldots,y_m\} \subseteq \{x_1,\ldots,x_m\},
    \end{equation*}
    the following formula works:
    \begin{equation*}
      \bigwedge_{j=1}^m \bigwedge_{k=1}^m \bigvee_{\ell=1}^m (y_j -
      y_k = x_\ell).
    \end{equation*}
    Probably the most complicated thing to express is the condition on
    simple roots, which can be written as
    \begin{gather*}
      \bigwedge_{j_0 = 1}^m \bigwedge_{j_1 = 1}^m \cdots \bigwedge_{j_{n_i - 2} = 1}^m ~ \bigvee_{k = 1}^m \exists z : z = -1 + x_k ~\wedge \\
      z^{n_i} + z^{n_i - 1} + y_{j_{n_i - 2}} z^{n_i - 2} + \cdots + y_{j_1} z + y_{j_0} = 0 \\
      \wedge~
      n_i z^{n_i - 1} + (n_i - 1)z^{n_i - 2} + (n_i - 2) y_{j_{n_i - 2}} z^{n_i - 3} + \cdots + 2 y_{j_2} z + y_{j_1} \ne 0.  \qedhere
    \end{gather*}
  \end{claimproof}
  For $0 \le i \le m$ let $\{e''_{i,1},\ldots,e''_{i,m}\}$ be an
  enumeration of $A''_i$ (possibly with repeats).  Then the following quantifier-free formula
  holds:
  \begin{gather*}
    \bigwedge_{i=0}^{m-1}
    \phi_i(e''_{i,1},\ldots,e''_{i,m};e''_{i+1,1},\ldots,e''_{i+1,m}) \\
    \wedge \bigwedge_{i=0}^m \theta_{A_i}(e''_{i,1},\ldots,e''_{i,m}) \\
    \wedge \theta'_{\{0\}}(e''_{n,1},\ldots,e''_{n,m}).
  \end{gather*}
  Indeed, this merely says
  \begin{gather*}
    \text{$C_i(\{e''_{i,1},\ldots,e''_{i,m}\},\{e''_{i+1,1},\ldots,e''_{i+1,m}\})$
      holds for each $i = 0, \ldots, m-1$} \\ \text{and $A_i \subseteq
      \{e''_{i,1},\ldots,e''_{i,m}\}$ for each $i$} \\ \text{and
      $\{0\} \subsetneq \{e''_{n,1},\ldots,e''_{n,m}\}$,}
  \end{gather*}
  and these things hold because $\{e''_{i,1},\ldots,e''_{i,m}\} =
  A''_i$.

  Now because $K$ is existentially closed in $R$, we can find
  $\{e'_{i,j}\}_{0 \le i \le m, ~ 1 \le j \le m}$ \emph{in $K$} such
  that
  \begin{gather*}
    \bigwedge_{i=0}^{m-1}
    \phi_i(e'_{i,1},\ldots,e'_{i,m};e'_{i+1,1},\ldots,e'_{i+1,m}) \\
    \wedge \bigwedge_{i=0}^m \theta_{A_i}(e'_{i,1},\ldots,e'_{i,m}) \\
    \wedge \theta'_{\{0\}}(e'_{n,1},\ldots,e'_{n,m}).
  \end{gather*}
  Let $A'_i = \{e'_{i,1},\ldots,e'_{i,m}\}$.  Then this says
  \begin{gather*}
    \text{$C_i(A'_i,A'_{i+1})$ holds for each $i = 0, \ldots, m-1$}
    \\ \text{and $A_i \subseteq A'_i$ for each $i$} \\ \text{and
      $\{0\} \subsetneq A'_n$.}
  \end{gather*}
  Taking $A'_i = A_i = \{0\}$ for $i > m$, we are done.
\end{proof}
\begin{theorem} \label{exists}
  If $K$ is large and countable, then there is a gt-henselian field
  topology on $K$.
\end{theorem}
\begin{proof}
  Let $A^0_i = \{0\}$ for each $i$.  Then $\{A^0_i\}_{i \in \Nn}$ is a
  suitable sequence (Remark~\ref{zero2}).  For $0 < j < \omega$, build
  an essentially finite suitable sequence $\{A^j_i\}_{i \in \Nn}$ in
  $K$ such that $A^j_i \supseteq A^{j-1}_i$ and $A^j_j \supsetneq
  A^{j-1}_j$, using Lemma~\ref{ec-trick} and induction on $j$.  Let
  $A^\omega_i = \bigcup_{j=0}^\infty A^j_i$ for each $i$.  Then
  $\{A^\omega_i\}_{i \in \Nn}$ is a suitable sequence in $K$ by
  Lemma~\ref{union-duh}.  Moreover, $A^\omega_j \supseteq A^j_j
  \supsetneq \{0\}$ for each $j \ge 1$, so no $A^\omega_i$ is $\{0\}$.
  By Lemma~\ref{obvious}, we get a gt-henselian topology.
\end{proof}


\section{The \'etale-open topology and gt-henselianity} \label{bonus}
Continue to fix a countable field $K$.  Recall that $\Ee$ denotes the \'etale-open topology.  Until Proposition~\ref{wave},
work in following setting:
\begin{assumption} \label{sump}
  $K$ is large.  $n$ is a positive integer.  $S \subseteq K^n$ is a
  set such that $\bar{0}$ is in the $\Ee$-closure of $S$, but
  $\bar{0} \notin S$.
\end{assumption}
Let $\Ll_P$ be the expansion of the language of rings by a new $n$-ary
relation symbol $P$.  An $\Ll_P$-structure is a pair $(R,P)$ where $R$
is an $\Ll_{rings}$-structure and $P \subseteq R^n$.
\begin{lemma} \label{bridge}
  There is an $\Ll_P$-structure $(R,S_R)$ extending $(K,S)$, with the
  following properties:
  \begin{enumerate}
  \item $(K,S)$ is existentially closed in $(R,S_R)$.
  \item $R$ is a henselian local domain with maximal ideal $\mm$ and
    residue field $K$, and $\mm$ is non-zero.
  \item $S_R$ contains an $n$-tuple $(a_1,\ldots,a_n)$ with each $a_i
    \in \mm$.
  \end{enumerate}
\end{lemma}
\begin{proof}
  By an \emph{affine \'etale neighborhood} of $\bar{0} \in \Aa^n(K)$,
  we mean a triple $(U,f,p)$ where $U$ is an affine $K$-variety, $f :
  U \to \Aa^n$ is an \'etale morphism, and $p \in U(K)$ is a
  $K$-rational point with $f(p) = \bar{0}$.  Recall that $K[U]$
  denotes the coordinate ring of $U$.  Affine \'etale neighborhoods form a
  filtered category in a natural way, and  \'etale local ring of
  $\Aa^n$ at $\bar{0}$ is
  \begin{equation*}
    \varinjlim_{(U,f,p)} K[U]
  \end{equation*}
  which can also be described as the henselization of the local ring
  $K[X_1,\ldots,X_n]_{(X_1,\ldots,X_n)}$ (see \cite[Lemmas~04GV,
    05KS]{stacks-project}).  We will write this henselization as
  $K[X_1,\ldots,X_n]^{\hens}$.
  \begin{claim}
    If $(U,f,p)$ is an affine \'etale neighborhood of $\bar{0}$, then
    there is a $K$-algebra homomorphism $K[U] \to K$ mapping the
    $n$-tuple $(X_1,\ldots,X_n)$ into the set $S$.
  \end{claim}
  \begin{claimproof}
    A homomorphism $K[U] \to K$ corresponds to a $K$-point $q \in
    U(K)$, and the composition
    \begin{equation*}
      K[X_1,\ldots,X_n] \to K[U] \stackrel{q}{\to} K
    \end{equation*}
    corresponds to the image $f(q) \in \Aa^n(K)$.  In particular, the
    image of the tuple $(X_1,\ldots,X_n)$ in $K$ is equal to the point
    $f(q) \in \Aa^n(K) = K^n$.

    Therefore, it suffices to find a point $q \in U(K)$ such that
    $f(q) \in S$.  The $K$-rational image $f(U(K)) \subseteq \Aa^n(K)$
    is an E-set in $\Aa^n(K)$ containing $\bar{0} = f(p)$, because
    $(U,f,p)$ is an \'etale neighborhood.  Then $f(U(K))$ is an
    $\Ee$-neighborhood of $\bar{0}$, so it intersects $S$.
  \end{claimproof}
  Let $(K^*,S^*)$ be a highly saturated elementary extension of
  $(K,S)$.
  \begin{claim} \label{dos}
    There is a $K$-algebra homomorphism $K[X_1,\ldots,X_n]^{\hens} \to
    K^*$ mapping the tuple $(X_1,\ldots,X_n)$ into $S^*$.
  \end{claim}
  \begin{claimproof}
    Because $K[X_1,\ldots,X_n]^{\hens}$ is a direct limit of
    coordinate rings $K[U]$ with $(U,f,p)$ an affine \'etale
    neighborhood of $\bar{0}$, this follows from the previous claim by
    saturation of $K^*$.
  \end{claimproof}
  Fix such a homomorphism $K[X_1,\ldots,X_n]^{\hens} \to K^*$ as in
  Claim~\ref{dos}.  Let $R \subseteq K^*$ be the image.  Then $R$ is a quotient of
  the henselian local ring $K[X_1,\ldots,X_n]^{\hens}$, so $R$ is
  itself a henselian local ring, and $K$ is the residue field of $R$.
  Let $S_R$ be the restriction of $S^*$ to $R$.  Each $X_i$ lies in
  the maximal ideal of $K[X_1,\ldots,X_n]^{\hens}$, so its image in
  $R$ lies in the maximal ideal $\mm$ of $R$.  Therefore, the image of
  the $n$-tuple $(X_1,\ldots,X_n)$ is an $n$-tuple in $\mm$ which lies
  in $S_R$.  This tuple must be non-zero, because $\bar{0} \notin S^*$;
  therefore $\mm$ is non-trivial.

  The embeddings $(K,S) \hookrightarrow (R,S_R) \hookrightarrow
  (K^*,S^*)$ and the elementary inclusion $(K,S) \preceq (K^*,S^*)$
  together imply that $(K,S)$ is existentially closed in $(R,S_R)$.
\end{proof}
Fix a structure $(R,S_R)$ and tuple $(a_1,\ldots,a_n)$ as in
Lemma~\ref{bridge}.  If we disregard the set $S_R$ and tuple $\ba$,
then we are in the setting of Remark~\ref{forge}.  Therefore,
Lemma~\ref{core} continues to apply in this setting.
\begin{lemma} \label{re-iterate}
  Let $A_0, A_1, A_2, \ldots$ be an essentially finite suitable
  sequence in $K$.  Given any $m$, there is an essentially finite
  suitable sequence $A'_0, A'_1, \ldots$ in $R$ such
  \begin{itemize}
  \item $A'_i
    \supseteq A_i$ for each $i$
  \item $A'_m \ne \{0\}$.
  \item The set $S_R$ contains an $n$-tuple from $A'_m$.
  \end{itemize}
\end{lemma}
\begin{proof}
  Just like the proof of Lemma~\ref{iterate}, except that we take
  $\{a_1,\ldots,a_n\}$ as the ``seed,'' i.e., setting $A'_m = A_m
  \cup \{a_1,\ldots,a_n\}$ and applying Lemma~\ref{core} to
  recursively choose $A'_{m-1}$, $A'_{m-2}$, \ldots, $A'_0$.  We can still
  apply Lemma~\ref{core} because the elements $a_1,\ldots,a_n$ come
  from the maximal ideal of $R$.  Since $\bar{0}$ isn't in $S$ or $S_R$, the tuple $\ba$ is nonzero, and so the requirement $A'_m \supseteq \{a_1,\ldots,a_n\}$ ensures that $A'_m \supsetneq \{0\}$.
\end{proof}
\begin{lemma} \label{re-ec-trick}
  Let $A_0, A_1, A_2, \ldots$ be an essentially finite suitable
  sequence in $K$.  Given any $m$, there is an essentially finite
  suitable sequence $A'_0, A'_1, \ldots$ in $K$ such
  \begin{itemize}
  \item $A'_i
    \supseteq A_i$ for each $i$
  \item $A'_m \ne \{0\}$.
  \item The set $S$ contains an $n$-tuple from $A'_m$.
  \end{itemize}
\end{lemma}
This is identical to Lemma~\ref{re-iterate}, but with $(R,S_R)$
replaced with $(K,S)$.
\begin{proof}
  As in the proof of Lemma~\ref{ec-trick}, this follows formally from
  Lemma~\ref{re-iterate}, because $(K,S)$ is existentially closed in $(R,S_R)$.  The formulas are now $\mathcal{L}_P$-formulas rather than $\mathcal{L}_{rings}$-formulas, however.
\end{proof}
\begin{proposition} \label{wave}
  Let $S$ be a subset of $K^n$ and $\bar{b} \in K^n$ be a point in the $\Ee$-frontier of $S$ (so $\bb$ is not in $S$, but $\bb$ is in the $\Ee$-closure of $S$).  Then there is a gt-henselian topology $\tau$ on
  $K$ such that $\bb$ is in the $\tau$-frontier of $S$.
\end{proposition}
\begin{proof}
	Note that the $\Ee$-topology is non-discrete because $S$ is not $\Ee$-closed, and so $K$ is large \cite[Theorem~C(1),
    Proposition~4.5]{firstpaper}.   Replacing $S$ with $S - \bar{b}$, we may
  assume $\bar{b} = \bar{0}$.  Then we are in the setting of
  Assumption~\ref{sump}.  As in the proof of Theorem~\ref{exists},
  build an increasing sequence of suitable sequences
  \begin{equation*}
    A^0_\bullet \subseteq A^1_\bullet \subseteq A^2_\bullet \subseteq \cdots
  \end{equation*}
  \emph{but using Lemma~\ref{re-ec-trick} rather than
  Lemma~\ref{ec-trick}}, to ensure that $S$ contains a tuple from
  $A^j_j$ for each $j$.  Then the limit sequence $A^\omega_\bullet$ is
  a suitable sequence such that $S$ contains a tuple from $A^\omega_j$
  for each $j$.  Since $S$ does not contain $\bar{0}$, it follows that
  no $A^\omega_j$ is $\{0\}$, and Lemma~\ref{obvious} shows that the
  family $\{A^\omega_j\}_{j \in \Nn}$ is a neighborhood basis for a
  gt-henselian topology $\tau$.  Because $S$ contains a tuple from
  every $A^\omega_j$, it follows that $\bar{0}$ is in the
  $\tau$-closure of $S$.
\end{proof}
\begin{corollary} \label{wave2}
	If $S \subseteq K^n$ is not closed in the $\Ee$-topology, then there is a gt-henselian topology $\tau$ on $K$ such that $S$ is not $\tau$-closed.
	\end{corollary}
\begin{proof}
	A set is non-closed iff its frontier is empty, so this follows directly from Proposition~\ref{wave}.
\end{proof}
At this point, the proof is nearly complete modulo machinery from
\cite{firstpaper,field-top-2}.  Recall the following definition from
\cite[Definition~1.2]{firstpaper}
\begin{definition}
  A \emph{system of topologies} over a field $K$ is a map assigning to
  each variety $V$ over $K$ a topology on $V(K)$ such that the
  following hold for any morphism $f : V \to W$ of $K$-varieties:
  \begin{enumerate}
  \item The map $V(K) \to
    W(K)$ is continuous.
  \item If $V \to W$ is a (scheme-theoretic) closed immersion, then
    $V(K) \to W(K)$ is a (topological) closed embedding.
  \item If $V \to W$ is a (scheme-theoretic) open immersion, then
    $V(K) \to W(K)$ is a (topological) open embedding.
  \end{enumerate}
\end{definition}
\begin{example}
  If $\tau$ is a field topology on $K$, then $\tau$ induces a topology
  on $V(K)$ for any $K$-variety $V$, and this gives a system of
  topologies over $K$.
\end{example}
In the following, we identify a field topology $\tau$ with its induced
system of topologies over $K$.
\begin{fact} \phantomsection \label{old-facts}
  \begin{enumerate}
  \item $\Ee$ is a system of topologies over $K$
    \cite[Theorem~A]{firstpaper}.
  \item \label{of2} If $\tau$ is a gt-henselian field topology on $K$,
    then $\tau$ refines $\Ee$
    \cite[Lemma~8.14]{field-top-2}.  In other words, any set which is
    open in the $\Ee$-topology is $\tau$-open.
  \end{enumerate}
\end{fact}
The next fact shows that a system of topologies over $K$ is determined
by what it does on $\Aa^n(K) = K^n$.
\begin{fact}[{\cite[Lemma~4.2]{firstpaper}}] \label{affine-reduction}
  Let $\mathcal{T}_1$ and $\mathcal{T}_2$ be two systems of topologies
  over $K$.  Suppose the $\mathcal{T}_1$-topology on $\Aa^n(K)$
  refines the $\mathcal{T}_2$-topology on $\Aa^n(K)$ for each $n$.
  Then $\mathcal{T}_1$ refines $\mathcal{T}_2$.
\end{fact}
\begin{definition}
  If $V$ is a variety over $K$, then the \emph{$\mathcal{I}$-topology
  on $V(K)$} is the intersection of the $\tau$-topologies as $\tau$
  ranges over gt-henselian field topologies on $K$.  In other words, a
  set $U \subseteq V(K)$ is $\mathcal{I}$-open if $U$ is $\tau$-open
  for every gt-henselian field topology on $K$.
\end{definition}
Note that if $K$ admits no gt-henselian topologies (e.g., $K$ is
non-large), then $\mathcal{I}$ is the discrete system of topologies.
\begin{theorem} \label{2m}
  If $K$ is countable, then $\mathcal{I} = \Ee$.  In other words, if
  $V$ is a $K$-variety, then a subset $U \subseteq V(K)$ is
  $\Ee$-open if and only if $U$ is $\tau$-open for every
  gt-henselian topology $\tau$ on $K$.
\end{theorem}
\begin{proof}
  By Fact~\ref{old-facts}(\ref{of2}), $\mathcal{I}$ refines
  $\Ee$.  By Corollary~\ref{wave2}, the
  $\Ee$-topology on $K^n = \Aa^n(K)$ refines the
  $\mathcal{I}$-topology on $K^n$.  By Fact~\ref{affine-reduction},
  $\Ee$ refines $\mathcal{I}$.
\end{proof}

\begin{acknowledgment}
  This paper grew out of conversations with Minh Tran, Jinhe Ye, and
  especially Erik Walsberg (who first suggested that
  Theorem~\ref{second-main} might be true).  The author was supported
  by the Ministry of Education of China (Grant No.\@ 22JJD110002).
\end{acknowledgment}

\bibliographystyle{plain} \bibliography{mybib}{}

\end{document}